\documentclass[a4paper]{amsart}

\usepackage{a4wide}
\usepackage{amsmath}
\usepackage{amssymb}
\usepackage{amstext}
\usepackage{amsthm}
\usepackage{graphicx}
\usepackage{caption}
\usepackage{subcaption}

\newcommand{\R}{\mathbb{R}} 
\newcommand{\N}{\mathbb{N}} 
\newcommand{\id}{\mathrm{I\hspace{-0.5ex}I}}
\newcommand{\M}{\R^{2d}}
\newcommand{\Z}{Z_t}
\newcommand{\V}{V_t}
\newcommand{\X}{X_t}

\newcommand{\bl}{\left(}
\newcommand{\br}{\right)}
\newcommand{\lb}{\left[}
\newcommand{\rb}{\right]}

\renewcommand{\P}{\mathcal{P}}
\newcommand{\dist}{\text{dist}}
\newcommand{\E}{{\mathrm{I\hspace{-0.5ex}E}}}
\newcommand{\K}{\mathcal{K}}

\newcommand{\J}{\mathcal{J}}
\newcommand{\C}{\mathcal{C}}

\newtheorem{theorem}{Theorem}[section]
\newtheorem{prop}{Proposition}[section]
\newtheorem{lemma}{Lemma}[section]

\theoremstyle{remark}
\newtheorem{remark}{Remark}[section]

\newtheorem{ass}{Assumption}

\begin{document}
\title[Trend to equilibrium for a delay mean-field equation]{Trend to equilibrium for a delay Vlasov--Fokker--Planck equation and explicit decay estimates}
\author{A. Klar, L. Kreusser\and O. Tse}

\maketitle

\begin{abstract}
In this paper, a  delay Vlasov--Fokker--Planck equation associated to a stochastic interacting particle system with delay is investigated analytically.
 Under certain restrictions on the parameters well-posedness and ergodicity of  the mean-field equation are shown and an  exponential rate of convergence towards the unique stationary solution is proven as long as the delay is finite. For infinte delay i.e., when all the history of the solution paths are taken into consideration polynomial decay of the solution is shown. 
\end{abstract}

\section{Introduction}

In recent years, the trend to equilibrium for solutions of kinetic and hyperbolic equations has been widely investigated by applying entropy dissipation methods, see \cite{Desvillettes,hypocoercivity:villani} and \cite{BZ,nat,Che,CG07,klar2015entropy} for instance. In \cite{Dolbeault2009511} and \cite{Dolbeault2015}, Dolbeault et al. proposed a simplified approach to prove the trend to equilibrium for a large class of linear equations. Starting from an interacting particle model, Bolley et al. used a probabilistic approach in \cite{Bolley2010} to show convergence towards equilibrium for solutions of the Vlasov--Fokker--Planck equation in Wasserstein distance with an exponential rate.

The present research is motivated by the consideration of 
models  describing the lay-down of fibers in textile production processes, which include fiber-fiber interactions. Such models have been recently introduced in \cite{klar} by adding the interaction of structures into a well-investigated model for nonwoven production processes \cite{GoetzKMW2007,BonillaGKMW2008, 3dfiber, ApproximateModels}. In \cite{GoetzKMW2007}  fibers are  interpreted as  paths of a stochastic differential equation with a projection of the velocity to the unit sphere. Besides, this model takes into account the finite size of fibers and prevents self-intersection, as well as intersection among fibers. On the microscopic level, this fiber lay-down process results in a large, coupled system of stochastic delay differential equations, while a delay mean-field equation may be derived on the mesoscopic scale (cf. \cite{klar}). 
In order to ensure a high quality of the web and the resulting nonwoven fabrics it is of great interest to consider how the solutions of these models converge to  equilibrium. In particular the speed of convergence is an important indicator of the quality and favorable properties of the nonwovens such as homogeneity and load capacity \cite{grothaus}. Since faster convergence indicates a more uniform production of the nonwovens, analysis of the speed of convergence allows the adjustment of process parameters in such a way that the optimal production process can be performed.
Adapting the approach from  \cite{Dolbeault2009511}, convergence for a non-interacting fiber model  at an exponential rate to a unique stationary state has been proven in \cite{ExponentialRateOfConvergenceToEquilibrium}. The trend to equilibrium for this non-interacting fiber model has also been analyzed in \cite{grothaus} by using Dirichlet forms and operator semigroup techniques.

Since the interacting fiber model is an extension of the Vlasov--Fokker--Planck equation, a similar approach as in \cite{Bolley2010} may be utilized for analyzing the trend to equilibrium for this model. However, special care needs to be taken in order to deal first with the delay interaction term and second with the  projection of the velocities to the unit sphere   which is different to the system considered in \cite{Bolley2010}.

In the present investigation we deal only with the first problem, i.e. we consider a
classical  interacting particle model as in  \cite{Bolley2010, Carrillo:2010:10.1142/S0218202510004684} and  include a delay term as in the above mentioned paper \cite{klar}.
 This delay mean-field equation may be seen as an extension of the classical Vlasov--Fokker--Planck equation.
 The full interacting fiber equations from \cite{klar} will be considered in forthcoming research.

The paper is organized as follows: In Section 2, we  introduce the microscopic delay system and its mean-field counterpart. Section 3 is devoted to a classical and the well-posedness of the mean field problem. Section 4 is concerned with an  estimate for the delay mean field equations and  the long-time behavior of the  mean-field  model for the case of finite delay. The final convergence theorem is stated and proven, and the dependence of the rate of convergence on the different parameters is discussed as well as the case of infinite delay.

\section{Microscopic and mean-field delay models}

Our starting point is an interacting particle model, see for example \cite{Bolley2010}.
As discussed in the introduction we interpret a solution path of the model as a fiber. 
Since fiber-fiber interactions will involve either the full or a portion of the solution path, the resulting model includes information of  the past of the process, i.e. a delay term.

To describe the models we are interested in, we introduce a notation based on the notation in \cite{golse} that will be used in the sequel. Let $d\geq 2$ denote the number of dimensions and let $\M$ be the state space of the fibers. The state $\Z=(\X,\V)\in\M$ of a fiber at time $t$ with initial condition $Z_0(z)=z$ is given by its position $\X$ and its velocity $\V$. 

\subsection{Microscopic delay system}
This section is devoted to the description of the delay models and some associated equations. Let $d\geq 2$ denote the number of spatial dimensions. We consider $N$ interacting fibers with position $\X^i\in\R^d$ and velocity $\V^i\in\R^d$ at each time $t\ge 0$ for $i=1,\ldots,N$. The interacting fiber model is formulated as a stochastic delay differential equation with state space $(\X^i,\V^i)\in \M$, given by
\begin{subequations}\label{eq:N interacting fibers}
\begin{align}
 d\X^i&= \V^i dt\\
 d\V^i&= -\nabla_x\Phi(\X^i)d t - \bl\frac{1}{N}\sum\nolimits_j\frac{1}{t}\int_{0}^{t} \nabla_x U(\X^i-X_s^j)d s\br dt -\gamma\V^i d t + \sqrt{2\sigma} d W_t^i,
\end{align}
\end{subequations}
subjected to the initial conditions $(X_0^i,V_0^i)\in\M$. Here, $W_t^i$ for $i=1, \ldots, N$, denote independent standard Brownian motions on $\R^d$ and $\gamma>0$ is the friction coefficient. The vectors $(X_0^i,V_0^i)\in\M$ for $i=1, \ldots, N$ are typically taken to be independent and identically distributed random variables with a given law $\mu_0\in\P_2(\M)$, where $\P_2(\M)$ denotes the space of Borel probability measures with finite second moment. 
Further, $V$ is a coiling potential, which is of confinement type, and $U$ is an interaction potential. Note that the model includes a delay interaction term describing the interaction of fibers with each other and with themselves on the whole fiber length. The scaling by the factor $1/N$ is known as weak coupling scaling, which naturally leads to a mean-field equation (cf. \cite{Golse2003, neunzert}). Further note that the non-retarded version of \eqref{eq:N interacting fibers}, i.e., without the delay term, is similar to models for swarming with roosting potential as discussed in \cite{Carrillo:2009:10.3934/krm.2009.2.363, Carrillo:2010:10.1142/S0218202510004684}.

In applications,  fibers have a finite length. This can be described by introducing a cut-off 
\begin{align*}\label{eq:cut-off h}
h(t)=\begin{cases} 
t & t\leq H \\ 
H & t>H \\ 
\end{cases},
\end{align*}
with the cut-off size $0\le H \le \infty$. In this case, the interacting fiber model becomes
\begin{subequations}\label{eq:N interacting general}
\begin{align}
 d\X^i&= \V^i dt\\
 d\V^i&= -\nabla_x\Phi(\X^i)d t - \bl\frac{1}{N}\sum\nolimits_j\frac{1}{h(t)}\int_{t-h(t)}^{t} \nabla_x U(\X^i-X_s^j)d s\br dt -\gamma\V^i d t+ \sqrt{2\sigma} d W_t^i.
\end{align}
\end{subequations}
Note that we obtain the non-retarded interacting particle model in the limit $H\to 0$, whereas taking the limit $H\to\infty$ results in our previous model \eqref{eq:N interacting fibers}.
In the following we will investigate the different cases.

\subsection{The mean-field equation}
In this section, we formally derive the mean-field equation corresponding to \eqref{eq:N interacting general}. The idea behind the mean-field limit lies in replacing the pairwise interaction term for each fiber $i=1,\ldots,N$ with
\[
 \frac{1}{h(t)}\int_{t-h(t)}^{t} \int_{\M} \nabla_x U(\X^i- \hat x)\, d\mu_s^{(N)}(d\hat x,d\hat v)\,ds,
\]
where $\mu_t^{(N)}$ is the stochastic empirical measure defined by
\[
 \mu_t^{(N)}(d\hat x,d\hat v)=\frac{1}{N}\sum\nolimits_{i=1}^N \delta_{(\X^i,\V^i)}(d\hat x,d\hat v).
\]
In the limit $N\to\infty$, one can show that sequence of stochastic empirical measures $\{\mu_t^{(N)}\}_{N\in\N}$ converges in law towards a deterministic limit $\mu_t\in\P_2(\M)$, which satisfies the retarded Vlasov--Fokker--Planck equation
\begin{equation}\label{eq:meanfield distribution}
 \partial_t \mu_t +v\cdot\nabla_x \mu_t+S[\mu_t] = \nabla_v\cdot\big(\sigma\nabla_v\mu_t + \gamma v\mu_t\big),\qquad \lim\nolimits_{t\searrow 0} \mu_t = \mu_0\in\P_2(\M),
\end{equation}
in the distribution sense, where the deterministic force term $S[\mu_t]$ is given by
\begin{align}\label{eq:S_operator}
 S[\mu_t] = -\nabla_x\Phi\cdot \nabla_v \mu_t - \nabla_v \cdot \left[\bl  \frac{1}{h(t)}\int_{t-h(t)}^t\int_{\R^d} \nabla U(\cdot-\hat x) \rho_t(d\hat x)\,d s \br \mu_t\right].
\end{align}
Here, the spatial density $\rho_t$ denotes the first marginal of $\mu_t$, i.e., 
\[
 \rho_t(B) = \mu_t(B,\R^d)\quad\text{for any Borel set}\quad B\in\mathcal{B}(\R^d).
\]

Furthermore, it is known that an underlying nonlinear stochastic process $\{Z_t=(\X,\V)\}$ exists for the solution $\mu_t$ of \eqref{eq:meanfield distribution}, which satisfies the Mc-Kean type stochastic differential equation
\begin{subequations}\label{eq:meanfield}
\begin{align}
 d\X &= \V dt\\
 d\V &= -\nabla_x\Phi(\X)d t - \bl\frac{1}{h(t)}\int_{t-h(t)}^{t} \int_{\R^d} \nabla_x U(\X-\hat x)\rho_s(d\hat x) \,d s\br dt -\gamma\V d t + \sqrt{2\sigma} d W_t,
\end{align}
\end{subequations}
with the initial condition $\mu_0=\text{law}(Z_0)\in\P_2(\M)$ and where $\mu_t = \text{law}(Z_t)$ for all $t\ge 0$.

\begin{remark}
 The microscopic model \eqref{eq:N interacting general} may be seen as a particle approximation of solutions to \eqref{eq:meanfield distribution} for a large number $N\in\N$ of $\R^{2d}$-valued processes $(\X^i,\V^i)$ for $i=1,\ldots,N$.
\end{remark}

\section{Preliminary definitions and results}

\subsection{The Halanay inequality}
Gronwall established an inequality that provides an estimate that bounds a function satisfying a certain differential equation. Halanay showed a similar result for equations with time lag, which is stated  in the following two propositions.  For proofs 
we refer to  Halanay \cite{halanay}.
The first theorem is  a comparison theorem.
\begin{prop}\label{prop halanay}
Let $f\colon\R\times\R\times\R\to \R$ be a function such that $f(t,u,v)$ is continuous for all $(u,v)$ and all $t_0\leq t<t_0+\xi$ for some $\xi>0$. Further, let $f(t,u,v)$ be increasing with respect to $v$. Supposing that $(dy/dt)(t)\leq f\bl t,y(t),\sup_{s\in[t-H,t]}y(s)\br$ for $t_0\leq t<t_0+\xi$ and some constant $H>0$. If $\phi$ is the solution of the equation
\[
 \frac{d\phi}{dt}(t) = f\Big( t,\phi(t),\sup\nolimits_{s\in[t-H,t]}\phi(s)\Big)\qquad \text{for $t\in [t_0,t_0+\xi)$},
\]
with the initial condition $\phi(s) = y(s)$, $s\in [t_0-H,t_0]$, then $y(t)\leq \phi(t)$ for $t_0\leq t <t_0+\xi$.
\end{prop} 

This is the used  to show a Gronwall type result that can be directly applied to obtain the desired  decay estimates in the case of finite cut-off. We give the proof for completeness.

\begin{prop}\label{lem:halanay}
Let $h$ denote a nonnegative, continuous and bounded function defined for $t\ge 0$ and $H=\sup_{t\ge 0}h(t)>0$. Further, let $y$ be a nonnegative function satisfying
\begin{align*}
\frac{d y}{d t}(t)\leq -a y(t)+b\sup\nolimits_{s\in[t-h(t),t]}y(s)\qquad \text{for $ t>t_0$},
\end{align*}
with initial data $y(s)=y(t_0)$, $s\in[t_0-H,t_0]$, where $a$ and $b$ are nonnegative constants satisfying $a>b\ge 0$. Then, $y$ may be estimated from above by
\begin{align*}
y(t)\leq y(t_0) \exp(-\lambda(t-t_0))\qquad \text{for all $t\geq t_0$},
\end{align*}
with the decay rate
\begin{align*}
 \lambda = a-\frac{1}{H}W(bH\exp(aH))>0,
\end{align*}
where $W$ is the product logarithm function, which satisfies $z=W(z)\exp(W(z))$ for any $z\in\R$.
\end{prop}
\begin{proof}
 Let $f\colon\R\to\R$, $\lambda\mapsto -a+\lambda+b \exp(\lambda H)$ and let
\begin{align}\label{eq: def phi}
\phi(t):=y(t_0)\exp(-\lambda(t-t_0))
\end{align}
where $\lambda$ satisfies $f(\lambda)=0$. We begin by showing that 
\[
 \lambda=a-\frac{1}{H}W(bH\exp(aH))
\]
is the unique zero of $f$. Indeed, simple calculations lead to
\begin{align*}
f(\lambda)&=-\frac{1}{H}W(bH\exp(a H)) + b\exp(a H)\exp(-W(bH\exp(a H)))\\&=-\frac{W\bl bH\exp\bl a H\br\br}{H}+b\exp\bl a H\br\bl\frac{bH\exp\bl a H\br}{W\bl b H\exp\bl a H\br\br}\br^{-1}=0,
\end{align*}
where we used the definition of $W$ in the second equality. Hence, $\lambda$ is indeed a zero of $f$. Since $f$ is strictly increasing and $f(0)=b-a<0$, $\lambda$ is positive and is the unique zero of $f$. 

Now consider the estimate
\[
 \frac{dy}{dt}(t) \leq -a y(t)+b\sup\nolimits_{s\in[t-h(t),t]}y(s)\leq -a y(t)+b\sup\nolimits_{s\in[t-H,t]}y(s),
\]
where the second inequality holds since $y$ is nonnegative, and the differential equation 
\begin{align}\label{eq:dgl help}
 \frac{d\phi}{dt}(t) = -a \phi(t)+b\sup\nolimits_{s\in[t-H,t]}\phi(s)\qquad \text{for $t>t_0$},
\end{align}
with initial data $\phi(s)=y(t_0)$ for $s\in[t_0-H,t_0]$. We show that $\phi$, as defined in \eqref{eq: def phi}, satisfies the differential equation \eqref{eq:dgl help}. Indeed, since by definition, we have
\begin{align*}
 -a \phi(t)+b\sup\nolimits_{s\in[t-H,t]}\phi(s)&= -a \phi(t)+b\phi(t-H)\\
 &=-a y(t_0)\exp(-\lambda(t-t_0)) + b y(t_0)\exp(-\lambda(t-H-t_0)) \\
 &=-\lambda y(t_0)\exp(-\lambda(t-t_0)) = \frac{d\phi}{dt}.
\end{align*}
Thus, applying Proposition~\ref{prop halanay} yields the estimate
\[
 y(t)\leq \phi(t)=y(t_0)\exp(-\lambda(t-t_0)) \qquad \text{for all $t\ge t_0$},
\]
thereby concluding the proof.
\end{proof}

\subsection{Metrics in Wasserstein space}

The set of Borel probability measures on the state space $\M$ is denoted by $\mathcal{P}(\M)$. Further, we denote by 
\[
\mathcal{P}_2(\M):=\left\lbrace \mu\in\P(\M) \quad\text{such that}\quad \int_{\M}|z|^2 \mu(dz)<\infty\right\rbrace,
\]
to be the set of Borel probability measures on $\M$ with finite second moment. This space may be equipped with the Wasserstein distance $\dist_2$ defined by 
\[
 \dist_2(\mu,\hat \mu)^2 =\inf_{\pi\in\Pi(\mu,\hat \mu)} \bl \int_{\M\times\M} |z-\hat z|^2\pi(d z,d\hat z)\br,\qquad \mu,\hat \mu\in\P_2(\M),
\]
where $\Pi(\mu,\hat \mu)$ denotes the collection of all Boral probability measures on $\M\times\M$ with marginals $\mu$ and $\hat \mu$ on the first and second factors respectively. The set $\Pi(\mu,\hat \mu)$ is also known as the set of all couplings of $\mu$ and $\hat\mu$. Equivalently, the Wasserstein distance may be defined by
\[
 \dist_2(\mu,\hat \mu)^2 = \inf \E\lb|Z-\hat Z|^2\rb,
\]
where the infimum is taken over all joint distributions of the random variable $Z$ and $\hat Z$ with marginals $\mu$ and $\hat \mu$ respectively.

Moreover, for a given positive quadratic form $Q\colon\M\to\R$, we define the distance, compare \cite{Bolley2010}
\begin{align*}
\dist_Q(\mu,\hat \mu)^2 = \inf_{\pi\in\Pi(\mu,\hat \mu)} \bl \int_{\M\times\M} Q\bl z-\hat z\br\pi(d z,d \hat z)\br =\E \lb Q(Z-\hat Z)\rb
\end{align*}
for any two probability measures $\mu,\hat \mu\in\P_2(\M)$.

\begin{remark}\label{rem:quadratic}
We note that 
 \begin{enumerate}
  \item $\dist_2$ metricizes the weak convergence in the Wasserstein space $\P_2(\M)$, as well as the convergence of the first two moments.  One can also show that $(\P_2(\M),\dist_2)$ is a complete metric space \cite{completemetricspace, optimaltransport:villani}.
  \item For $Q(z)=|z|^2$, we have that $\dist_Q=\dist_2$.
  \item For any $z\in\M$, let $z=(z_1,z_2)$ with $z_i\in\R^d$, $i\in\{1,2\}$. We define the quadratic form
  \[
   Q(z) = a|z_1|^2 + 2\langle z_1,z_2\rangle + b|z_2|^2,
  \]
  with positive constants $a,b>0$, where $\langle\cdot,\cdot\rangle$ denotes the euclidean scalar product in $\R^d$. If the constants $a,b$ satisfy, additionally, the inequality $ab>1$, then there exists $p,q>0$ such that $\dist_2$ and $\dist_Q$ are equivalent. More specifically, since
  \[
   p|z|^2 \le Q(z) \le q|z|^2,
  \]
  holds for all $z\in\M$, with
  \[
   p = \frac{(a + b) - \sqrt{4 + (b-a)^2}}{2},\qquad q =\frac{(a + b) + \sqrt{4 + (b-a)^2}}{2},
  \]
  integrating over $\pi\in\Pi(\mu,\hat \mu)$ and minimizing over all coupling yields the required assertion.
 \end{enumerate}
\end{remark}

\begin{lemma}\label{lem:quadratic}
 Let $Q$ be a positive quadratic form on $\M$ as in Remark~\ref{rem:quadratic}(3), i.e.,
 \[
  Q(z) = a|z_1|^2 + 2\langle z_1,z_2\rangle + b|z_2|^2,\qquad a,b>0.
 \]
 If further $ab>1$, then $Q$ may be represented as
 \[
  Q(z) = \langle z, M_Qz\rangle = \langle Sz,Sz \rangle,
 \]
 where $M_Q$ and $S\in\text{Mat}\,(\M)$ are invertible block matrices of the form
 \[
  M_Q = \begin{bmatrix}
   a\,\id_d & \id_d \\
   \id_d & b\,\id_d
  \end{bmatrix},\qquad S=\frac{1}{\sqrt{a}}\begin{bmatrix}
   a\,\id_d & \id_d \\[1ex]
   0 & \sqrt{ab-1}\,\id_d
  \end{bmatrix},
 \]
 with $\id_d$ being the identity matrix on $\R^d$.
\end{lemma}

\subsection{Well-posedness of the delay mean-field equation}

In the following, we will make use of the fact that an underlying stochastic process $\{Z_t=(\X,\V)\}$ exists for $\mu_t$ satisfying \eqref{eq:meanfield}. In fact, we will consider a more general form of our mean-field equation, namely the problem, compare again \cite{Bolley2010} for the case without delay
\begin{subequations}\label{eq:general_meanfield}
\begin{align}
 d\X &= \V dt\\
 d\V &= A(\X) dt + \bl\frac{1}{h(t)}\int_{t-h(t)}^{t} \int_{\R^d} B(\X,\hat x)\rho_s(d\hat x) \,d s\br dt -\gamma\V d t + \sqrt{2\sigma} d W_t,
\end{align}
\end{subequations}
where $A$ and $B$ satisfies the following assumptions:
\begin{ass}\label{ass}
\begin{enumerate}
  \item We assume that $A\colon\R^d\to\R^d$ is of the form
  \[
   A(x) = -\alpha x + g(x)
  \]
  for some $\alpha>0$ where $g$ is Lipschitz continuous with Lipschitz constant $c_g>0$.
  \item The function $B\colon \R^d\times\R^d\to\R^d$ satisfies 
  \begin{align*}
   |B(x,\hat x) - B(y,\hat x)| + |B(x,\hat x) - B(x,\hat y)| \le c_B\big( |x-y| + |\hat x - \hat y| \big),
  \end{align*}
  for some constant $c_B>0$, independent of $x,\hat x,y,\hat x\in\R^d$.
\end{enumerate}
\end{ass}

\begin{remark}
 Without loss of generality, we may further assume that $\alpha\equiv 1$, which may be justified by a rescaling of time $\tau=\sqrt{\alpha}\,t$, as in the case of a simple harmonic oscillator.
\end{remark}

The next proposition provides a well-posedness result for problem \eqref{eq:general_meanfield}. This result generalizes the deterministic result provided in \cite{klar}  to the stochastic case. We note that \cite{klar} the velocities are projected to the unit sphere in contrast to the present case.
The arguments of the proof follow those made in \cite{Golse2012}  and \cite{klar}.
We include the proof in Appendix~\ref{proof:wellposedness} for the sake of completeness.

\begin{prop}\label{prop:wellposedness}
 Under Assumption~\ref{ass}, there exists a unique solution $\mu\in\C([0,T],\P_2(\M))$ of the mean-field equation
\begin{equation}
 \partial_t \mu_t +v\cdot\nabla_x \mu_t+S[\mu_t] = \nabla_v\cdot\big(\sigma\nabla_v\mu_t + \gamma v\mu_t\big),\qquad \lim\nolimits_{t\searrow 0} \mu_t = \mu_0\in\P_2(\M),
\end{equation}
in the distribution sense, where the deterministic force term $S[\mu_t]$ is given by
\[
 S[\mu_t] = A\cdot \nabla_v \mu_t + \nabla_v \cdot \left[\bl  \frac{1}{h(t)}\int_{t-h(t)}^t\int_{\R^d} B(\cdot,\hat x) \rho_t(d\hat x)\,d s \br \mu_t\right].
\]
Equivalently, there exists a pathwise unique nonlinear process $\{\Z=(\X,\V),\,t\in[0,T]\}$ satisfying \eqref{eq:general_meanfield} with initial data $\mu_0=\text{law}(Z_0)\in\P_2(\M)$ and $\mu_t = \text{law}(Z_t)$ for all $t\in[0,T]$.
\end{prop}

\section{Long time behavior of the delay Vlasov--Fokker--Planck equation}
In this section, we prove a quantitative exponential convergence result for all solutions of \eqref{eq:meanfield} with a finite cut-off 
$0\le H < \infty$ associated to the microscopic system
(\ref{eq:N interacting general}). The proof is based on the idea of perturbing the Euclidean metric on $\M$ in such a way that \eqref{eq:meanfield distribution} is completely dissipative with respect to the new metric. This idea is also present in \cite{Bolley2010, stochasticHamiltonianDissipativeSystems, hypocoercivity:villani}.

\subsection{Decay estimates and explicit rates}

The key estimate  for proving the exponential convergence of all solutions to a unique equilibrium is given in the following proposition. 
\begin{theorem}\label{prop:decay}
Under Assumption~\ref{ass}, for any $\gamma>0$ and $\alpha\equiv 1$, there exists a constants $\lambda,\eta_0>0$ such that for any $c_g,c_B\geq 0$ with $\eta:=c_g+2c_B\in[0,\eta_0)$, the decay estimate
\begin{align}\label{eq:ineq dq}
 \dist_Q(\mu_t,\hat{\mu}_t)\leq \exp(-\lambda t)\dist_Q(\mu_0,\hat{\mu}_0)\qquad t\ge 0,
\end{align}
holds true for any solution $(\mu_t)_{t\geq 0}$, $(\hat{\mu}_t)_{t\geq 0}$ of \eqref{eq:general_meanfield} with corresponding initial data $\mu_0,\hat{\mu}_0\in\P_2(\M)$, where $Q$ is a positive quadratic form on $\M$ of the form
\[
 Q(x,v) = a|x|^2+2\langle x,v\rangle + b|v|^2,\qquad a,b>0.
\]
Furthermore, the decay rate $\lambda$ is explicitly given by
\begin{align}\label{eq:rate}
\lambda = \lambda_1 - \frac{1}{H}W(\lambda_2 H\exp(\lambda_1 H)),
\end{align}
where $\lambda_i$, $i=1,2$ are given in \eqref{eq:lambdas}.
\end{theorem}

The idea of the proof is to find an appropriate positive quadratic form $Q$, such that a differential inequality appears. Applying the Halanay inequality (cf.~Lemma~\ref{lem:halanay}) onto this differential inequality would then yield the required decay estimate \eqref{eq:ineq dq}.

\begin{proof}
Let $\{\Z,\,t\ge 0\}$ and $\{\hat\Z,\,t\ge 0\}$ be two solutions of \eqref{eq:general_meanfield} corresponding to the distributions $\mu$ and $\hat\mu\in\C([0,\infty),\P_2(\M))$ respectively. Supposing that both of these solutions satisfy the equations with the same Brownian motion $\{W_t\}_{t\ge 0}$ in $\R^d$ and denoting the difference of the solutions by $x_t = \X-\hat \X$, $v_t = \V - \hat \V$, we have that $z_t=(x_t,v_t)\in\M$ satisfies
 \begin{subequations}\label{eq:difference}
 \begin{align}
  dx_t &= v_t dt\\
  dv_t &= \left[(A(\X) - A(\hat \X) ) - \frac{1}{h(t)}\int_{t-h(t)}^{t} \big(\K[\mu_s](\X) - \K[\hat\mu_s](\hat\X)\big)\,d s -\gamma v_t \right] d t,
 \end{align}
 \end{subequations}
 where for ease of notation, we define the operator $\K$ acting on probability measures as
 \[
  \K[\mu_t](x) := \int_{\R^d} B(x, y)\,\rho_t(dy).
 \]
Before proceeding, we compute the evolution of several functionals that will be useful.
\begin{align*}
 \frac{d}{dt}|x_t|^2 &= 2\langle x_t,v_t\rangle \\
 \frac{d}{dt}|v_t|^2 &= 2\left[\langle v_t, A(\X) - A(\hat \X)\rangle - \frac{1}{h(t)}\int_{t-h(t)}^{t} \langle v_t,\K[\mu_t](\X) - \K[\hat\mu_t](\hat\X)\rangle d s - \gamma |v_t|^2\right]\\
 \frac{d}{dt}\langle x_t,v_t\rangle &= |v_t|^2 + \langle x_t,A(\X)-A(\hat \X)\rangle \\
 &\hspace{10em}- \frac{1}{h(t)}\int_{t-h(t)}^{t} \langle x_t,\K[\mu_t](\X) - \K[\hat\mu_t](\hat\X)\rangle ds - \gamma \langle x_t,v_t\rangle.
\end{align*}
Now define the functional
\[
 \J_t:= \E\lb Q(z_t) \rb = \E\lb a|x_t|^2 + 2\langle x_t,v_t\rangle + b|v_t|^2 \rb,
\]
where $a,b>0$ are appropriately chosen later. From the previous computations, we obtain
\begin{align*}
 \frac{d}{dt}\J_t &= -2\,\E\lb |x_t|^2 - (a-b-\gamma)\langle x_t,v_t\rangle + (b\gamma - 1)|v_t|^2 \rb \\
 &\hspace*{10em} + 2\,\E\lb \langle x_t,g(\X) - g(\hat \X)\rangle + b\langle v_t,g(\X) - g(\hat \X)\rangle \rb \\
 &\hspace*{2em} - \frac{1}{h(t)}\int_{t-h(t)}^{t} 2\,\E\lb \langle x_t,\K[\mu_t](\X) - \K[\hat\mu_t](\hat\X)\rangle + b\langle v_t,\K[\mu_t](\X) - \K[\hat\mu_t](\hat\X)\rangle \rb ds.
\end{align*}
Setting $\pi_t = \rho_t\otimes\hat\rho_t$, we estimate the terms separately to obtain
\begin{gather*}
 \langle x_t,g(\X) - g(\hat \X)\rangle \le c_g|x_t|^2,\qquad \langle v_t,g(\X) - g(\hat \X)\rangle\le c_g\lb\frac{\delta_1}{2}|x_t|^2 + \frac{1}{2\delta_1}|v_t|^2\rb,
\end{gather*}
\begin{align*}
 \langle x_t,\K[\mu_s](\X) - \K[\hat\mu_s](\hat\X)\rangle &= \iint_{\R^d\times\R^d} \langle x_t, B(\X,y)-B(\hat\X,\hat y)\rangle\, \pi_s(dy,d\hat y) \\
 &\le c_B\lb \bl 1 + \frac{\delta_2}{2}\br|x_t|^2 + \frac{1}{2\delta_2}\iint_{\R^d\times\R^d} |y-\hat y|^2\, \pi_s(dy,d\hat y)\rb,
\end{align*}
\begin{align*}
  \langle v_t,\K[\mu_s](\X) - \K[\hat\mu_s](\hat\X)\rangle &= \iint_{\R^d\times\R^d} \langle v_t, B(\X,y)-B(\hat\X,\hat y)\rangle\, \pi_s(dy,d\hat y) \\
  &\hspace{-4em}\le c_B\lb\frac{\delta_3}{2}|x_t|^2 + \frac{1}{2}\bl\frac{1}{\delta_3} + \delta_4  \br|v_t|^2 + \frac{1}{2\delta_4}\iint_{\R^d\times\R^d} |y-\hat y|^2\, \pi_s(dy,d\hat y)\rb.
\end{align*}
Putting all these terms together yields
\begin{align*}
 \frac{d}{dt}\J_t &\le -2\bl 1 - \bl c_g + c_B + \frac{b}{2}(c_g\delta_1 + c_B\delta_3) + c_B\frac{\delta_2}{2} \br\br  \E\lb|x_t|^2\rb - 2(a- b-\gamma)\E\lb \langle x_t,v_t\rangle \rb \\
 &\hspace*{-1em}- 2\bl b\gamma - 1 - \frac{b}{2}\bl\frac{c_g}{\delta_1} + \bl \frac{1}{\delta_3} + \delta_4\br c_B \br\br\E\lb|v_t|^2 \rb + c_B\bl \frac{1}{\delta_2} + b\frac{1}{\delta_4} \br \frac{1}{h(t)}\int_{t-h(t)}^t \E\lb |x_s|^2\rb ds.
\end{align*}
Choosing $a= b+\gamma$, $\delta_1=\delta_3=\delta_4=1$ and $\delta_2=2 + b$, we obtain
\begin{align*}
  \frac{d}{dt}\J_t &\le -\big( 2 - \bl 2 + b \br\eta\big)  \E\lb|x_t|^2\rb - \big( 2b\gamma - 2 - b\eta\big) \E\lb|v_t|^2 \rb \\
  &\hspace*{16em}+ \frac{\eta}{2}\bl b + \frac{1}{2+b} \br \frac{1}{h(t)}\int_{t-h(t)}^t \E\lb |x_s|^2\rb ds.
\end{align*}
where $\eta=c_g + 2c_B$ and we used the fact that $c_B \le \eta/2$ in the last term.

Recall from Lemma~\ref{lem:quadratic} that the quadratic form $Q$ may be represented as
\[
 Q(z) = \langle Sz, Sz\rangle =: \langle \xi,\xi\rangle,\qquad S=\frac{1}{\sqrt{ b+\gamma}}\begin{bmatrix}
    ( b+\gamma)\,\id_d & \id_d \\[1ex]
    0 & \sqrt{b( b+\gamma)-1}\,\id_d
   \end{bmatrix}.
\]
Furthermore, the first two terms on the right hand side may be equivalently written as
\[
 \big( 2 - \bl 2 + b \br\eta\big)  \E\lb|x_t|^2\rb + \big( 2b\gamma - 2 - b\eta\big) \E\lb|v_t|^2\rb  = \E\lb \langle\xi_t, D_1\xi_t\rangle \rb,
\]
with the matrix
\[
 D_1 = d_1\begin{bmatrix}
    \id_d & -d_2\id_d  \\
    -d_2\id_d  & d_2^2(1+d_3)\,\id_d
   \end{bmatrix},
\]
where $d_i>0$, $i=1,2,3$, are given explicitly by
\[
 d_1 = \frac{2 - (2+b)\eta}{ b + \gamma},\qquad d_2 = \frac{1}{\sqrt{b( b+\gamma)-1}},\qquad d_3 = \frac{ b+\gamma}{d_1} (2b\gamma -2 -b\eta).
\]
Similarly, the last term may be reformulated to read
\[
 \frac{\eta}{2}\bl b + \frac{1}{2+b} \br \frac{1}{h(t)}\int_{t-h(t)}^t \E\lb |x_s|^2\rb ds = \frac{1}{h(t)}\int_{t-h(t)}^t \E\lb \langle \xi_t, D_2 \xi_t\rangle \rb ds,
\]
with 
\[
 D_2 = d_4\begin{bmatrix}
    \id_d & -d_2\id_d \\
    -d_2\id_d & d_2^2\,\id_d
   \end{bmatrix},\qquad d_4 = \frac{(1+b)^2}{(2+b)( b+\gamma)}\frac{\eta}{2}.
\]
At this point, we clearly see the conditions on the parameters $b,\gamma$ and $\eta$, such that $Q$ becomes a positive quadratic form. More concretely, we require that
\[
 b( b + \gamma) > 1,\qquad 2 > (2+b)\eta,\qquad 2b\gamma > 2 + b\eta,
\]
which guarantees that $d_i>0$. Reformulating the inequalities for $\eta$ and $b$ provide the conditions
\[
 0<\eta < 1 + \gamma - \sqrt{1 + \gamma^2},\qquad \frac{1}{2\gamma-\eta} < \frac{b}{2} < \frac{1-\eta}{\eta}.
\]
Continuing where we left off, we reformulate the inequality for $\J_t$ in terms of $\xi_t$ to obtain
\begin{align*}
 \frac{d}{dt}\J_t &\le - d_1\Big( \E\lb |\xi_{1,t}|^2 \rb - 2d_2\E\lb\langle \xi_{1,t},\xi_{2,t}\rangle\rb + d_2^2(1+d_3)\E\lb|\xi_{2,t}|^2 \rb \Big) \\
 &\hspace{6em}+ \frac{1}{h(t)}\int_{t-h(t)}^t d_4\Big(\E\lb |\xi_{1,s}|^2 \rb - 2d_2\E\lb\langle \xi_{1,s},\xi_{2,s}\rangle\rb + d_2^2\E\lb|\xi_{2,s}|^2 \rb\Big) ds \\
 &\le - d_1\Big( (1-\epsilon)\E\lb |\xi_{1,t}|^2 \rb + d_2^2(1+d_3-1/\epsilon)\E\lb|\xi_{2,t}|^2 \rb \Big) \\
  &\hspace{16em}+ \frac{1}{h(t)}\int_{t-h(t)}^t d_4(1+d_2^2)\,\E\lb |\xi_s|^2 \rb ds.
\end{align*}
We now choose $\epsilon>0$ such that $(1-\epsilon) = (1-d_3-1/\epsilon)d_2^2$, which gives
\[
 \epsilon = \frac{1}{2}\bl 1-(1+d_3)d_2^2 + \sqrt{4d_2^2 + (1-(1+d_3)d_2^2)^2} \br.
\]
Consequently, we obtain 
\begin{align}\label{eq:diff_inequality_int}
 \frac{d}{dt}\J_t \le -\lambda_1 \J_t + \lambda_2
 \frac{1}{h(t)} \int_{t-h(t)}^t\J_s ds,
\end{align}
with
\[
 \lambda_1 = \frac{d_1}{2}\bl 1 + (1+d_3)d_2^2 - \sqrt{4d_2^2 + (1-(1+d_3)d_2^2)^2} \br,\qquad \lambda_2 = d_4(1+d_2^2).
\]
This gives the differential inequality
\begin{align}\label{eq:diff_inequality}
 \frac{d}{dt}\J_t \le -\lambda_1 \J_t + \lambda_2\sup\nolimits_{s\in[t-h(t),t]} \J_s,
\end{align}
The goal, now, is to find $b>0$, such that $\lambda_1$ is maximized. To do so, we begin by setting $\eta\equiv 0$. In this case, we have that
\[
 d_1 = 2/(b+\gamma),\qquad d_3 = (b+\gamma)^2(b\gamma-1),\qquad \lambda_2\equiv 0,
\]
which consequently yields,
\[
 \lambda_1 = \gamma - \frac{1}{b+\gamma}\sqrt{4d_2^2 + (2-\gamma(b+\gamma))^2} = \gamma - \sqrt{\frac{(2-b\gamma)^2 + (b\gamma-1)\gamma^2}{b(b+\gamma)-1}},
\]
Therefore, maximizing in $b$ for $\lambda_1$ gives
\begin{align}
 b = \frac{2}{\gamma},\qquad \lambda_1 = \gamma\bl 1 - \sqrt{\frac{\gamma^2}{4+\gamma^2}}\,\br.
\end{align}
With this particular choice of $b>0$, we proceed to find conditions on $\eta>0$ such that $\lambda_1>\lambda_2$ is satisfied. Assuming that
\[
 (\gamma(1-\eta)-\eta) > 0 \quad\Longleftrightarrow\quad 0<\eta<\gamma/(1+\gamma),
\]
we have that $\lambda_i>0$, where $\lambda_i$ for $i=1,2$ are given by
\begin{align}\label{eq:lambdas}
 \lambda_1 = \gamma -\bl 1 + \frac{2\gamma}{4+\gamma^2}\br\eta - \frac{\gamma}{4+\gamma^2}\sqrt{4\eta^2 + (4+\gamma^2)(\gamma-\eta)^2},\qquad \lambda_2 = \frac{(2+\gamma)^2}{(1+\gamma)(4+\gamma^2)}\frac{\eta}{2}.
\end{align}
To further ensure that $\lambda_1>\lambda_2$, we may choose $\eta$ satisfying the inequality
\[
 0<\eta \le \frac{2\gamma}{3(1+\gamma)}<\eta_0,
\]
for some $\eta_0>0$. Finally, we apply Lemma~\ref{lem:halanay} onto \eqref{eq:diff_inequality} to obtain
\begin{align}\label{eq:J_inequality}
 \dist_Q(\mu_t,\hat\mu_t)\le\E\lb Q(\Z-\hat \Z) \rb \le \E\lb Q(Z_0-\hat Z_0)\rb \exp(-\lambda t)\qquad\text{for $t\ge 0$},
\end{align}
with
\[
 \lambda = \lambda_1 - \frac{1}{H}W(\lambda_2 H\exp(\lambda_1 H)).
\]
Optimizing \eqref{eq:J_inequality} over all joint distributions of the random variable $Z_0$ and $\hat Z_0$ with marginals $\mu_0$ and $\hat\mu_0$ respectively provides the required estimate \eqref{eq:ineq dq}, thereby concluding the assertion.
\end{proof}

\subsection{Exponential convergence to a unique equilibrium}

As in \cite{Bolley2010}, the following lemma guarantees the existence of a unique equilibrium in a complete metric space. It is taken from \cite{carrillo}.
\begin{lemma}[cf.~\cite{carrillo}]\label{lem unique equilibrium}
Let $(X,d)$ be a complete metric space and $T_t\colon (X,d)\to (X,d)$ be a continuous semigroup for which, for all $t>0$, there exists $0<L(t)<1$ such that
\begin{align*}
d(T_t(x),T_t (y))\leq L(t)d(x,y),\quad t>0,\quad x,y\in X.
\end{align*}
Then there exists a unique stationary point $x_{\infty}\in X$, i.e., $T_t(x_{\infty})=x_{\infty}$ for all $t>0$.
\end{lemma}

Using Lemma \ref{lem unique equilibrium}, we can now prove the exponential convergence of all solutions to a unique equilibrium under Assumption~\ref{ass} and for sufficiently small $c_g$ and $c_B$.

\begin{theorem}\label{thm:limit}
Under the assumptions of Proposition~\ref{prop:decay}, there exists a constant $c>0$ such that
\begin{align}\label{eq:dist_2}
\dist_2\left(\mu_t,\hat{\mu}_t\right)\leq c\exp(-\lambda t)\dist_2 \left(\mu_0,\hat{\mu}_0\right)\qquad \text{for $t\geq 0$},
\end{align}
for all solutions $(\mu_t)_{t\geq 0}$ and $(\hat{\mu}_t)_{t\geq 0}$ to \eqref{eq:meanfield distribution} with initial data $\mu_0,\hat{\mu}_0\in\P_2(\M)$, respectively. 

Moreover, \eqref{eq:meanfield distribution} has a unique stationary solution $\mu_{\infty}\in\P_2(\M)$ such that
\begin{align}\label{eq:limit}
\dist_2\left(\mu_t,\mu_{\infty}\right)\leq c\exp(-\lambda t)\dist_2 \left(\mu_0,\mu_{\infty}\right)\qquad \text{for $t\geq 0$},
\end{align}
for any solution $(\mu_t)_{t\geq 0}$ of \eqref{eq:meanfield distribution}, where the decay rate $\lambda$ is as given in Proposition~\ref{prop:decay}.
\end{theorem}
\begin{proof}
 From Proposition~\ref{prop:decay} we obtain a positive quadratic form 
 \[
  Q(z)=a|z_1|^2 + 2\langle z_1,z_2\rangle + b|z_2|^2,\qquad a,b>0,\;ab>1,
 \]
 such that \eqref{eq:ineq dq} holds. Consequently, the metrics $\dist_Q$ and $\dist_2$ are equivalent in $\P_2(\M)$ (cf.~\ref{rem:quadratic}), which provides a constant $c>0$ such that \eqref{eq:dist_2} holds true. 
 
 Owing to the equivalence of $\dist_Q$ and $\dist_2$, we find that $(\P_2(\M),\dist_Q)$ is a complete metric space (cf.~Remark~\ref{rem:quadratic}). Further, $T_t\mu_0 := \mu_t$ defines a continuous semigroup on $(\P_2(\M),\dist_Q)$ that satisfies \eqref{eq:ineq dq}, which provides for a contraction, for $t>0$. Therefore, Lemma~\ref{lem unique equilibrium} guarantees the existence of a unique stationary solution $\mu_\infty\in\P_2(\M)$, thereby concluding the proof.
\end{proof}

\begin{remark}
 According to the previous result, we recover the stationary state $\mu_\infty$ when passing to the limit $t\to\infty$, which satisfies the stationary equation
 \begin{align*}
  v\cdot\nabla_x \mu_\infty + S[\mu_\infty] = \nabla_v\cdot\big(\sigma\nabla_v\mu_\infty + \gamma v\mu_\infty\big),
 \end{align*}
 where $S$ is defined by \eqref{eq:S_operator}, which in the stationary case, yields
 \[
  S[\mu_\infty] = -\nabla_x\Phi\cdot \nabla_v \mu_\infty - \nabla_v \cdot \left[\bl  \int_{\M} \nabla U(\cdot-\hat x) \mu_\infty(d\hat x,d\hat v) \br \mu_\infty\right].
 \]
 Using an Ansatz function of the form
 \begin{align}\label{eq:stationary ansatz}
 \mu_\infty(x,v)=\frac{1}{\sqrt{2\pi\vartheta^2}^d}\exp\left(-\frac{|v|^2}{2\vartheta^2}\right)\rho_\infty(x),\qquad \int_{\R^d}\rho_\infty\,dx=1,
 \end{align}
 for some function $\rho_\infty$ and a parameter $\theta$ it is easy to verify that 
 \begin{align*}
 \nabla_v\cdot\big(\sigma\nabla_v\mu_\infty + \gamma v\mu_\infty\big)=0 \quad\text{for} \quad \vartheta=\sqrt{\sigma/\gamma}.
 \end{align*}
 Hence it suffices to look for a solution of  
 \begin{align}\label{eq:stationary help eq}
   v\cdot\nabla_x \mu_\infty + S[\mu_\infty] =0. 
 \end{align}
 Inserting the ansatz \eqref{eq:stationary ansatz} into \eqref{eq:stationary help eq} and reformulating $\nabla_x \rho_\infty =\rho_\infty\nabla_x \log \rho_\infty$ results in the equation
 \[
 \rho_\infty\nabla_x\left(\vartheta^2\log \rho_\infty + V + U\star\rho_\infty\right)=0,
 \]
 where $\star$ denotes the convolution operator in $x$. This leads to the integral equation
 \[
  \vartheta^2\log \rho_\infty + V + U\star\rho_\infty = c,
 \]
 where the constant $c$ is uniquely determined by the normalizing constraint $\int_{\R^d}\rho_\infty\,dx=1$. 
 
 Note that the integral equation may be expressed in the equivalent fixed point form
 \[
 \rho_\infty =\frac{e^{-(V+U\star\rho_\infty)/\vartheta^2}}{\int_{\R^d}e^{-(V+U\star\rho_\infty)/\vartheta^2} dx}.
 \]
 If additionally, $U(x)=U(-x)$, then $\rho_\infty$ may be characterized as a minimizer of the functional
 \[
  \mathcal{F}(\rho_\infty) = \int_{\R^d} \vartheta^2\big(\log \rho_\infty -1\big)\rho_\infty(dx) + \int_{\R^d} V\rho_\infty(dx) + \frac{1}{2}\int_{\R^d} (U\star\rho_\infty )\rho_\infty(dx),
 \]
 where the first term describes the internal energy (entropy), the second the potential energy, and the third the interaction energy.
\end{remark}

\subsection{Discussion on decay rates} Without loss of generality, we only consider the case $\alpha\equiv 1$. Rewriting the decay rate \eqref{eq:rate} as
\[
 \lambda(\hat H) = \lambda_1\big( 1 - W(\hat \lambda \hat H\exp(\hat H))/\hat H \big),
\]
with $\hat H = \lambda_1 H$ and $\hat \lambda = \lambda_2/\lambda_1$, we see that $\lambda$ is monotone decreasing with respect to $\hat H$, with
\[
 \lambda(0) = \lambda_1(1-\hat\lambda)=\lambda_1-\lambda_2,\qquad \lim\nolimits_{\hat H\to\infty}\lambda(\hat H) = 0.
\]
In Figure \ref{fig:ineq} we plot the constraint on $\eta$ for a given fixed value of $\gamma$. More precisely, we plot the domain  
$\eta\le \bar\eta(\gamma):=2\gamma/(3+3\gamma)$, which satisfies $\eta\in[0,\eta_0)$, with $\eta_0>0$ given in Theorem~\ref{prop:decay}.

\begin{figure}[h]\centering
	\includegraphics[width=0.5\textwidth]{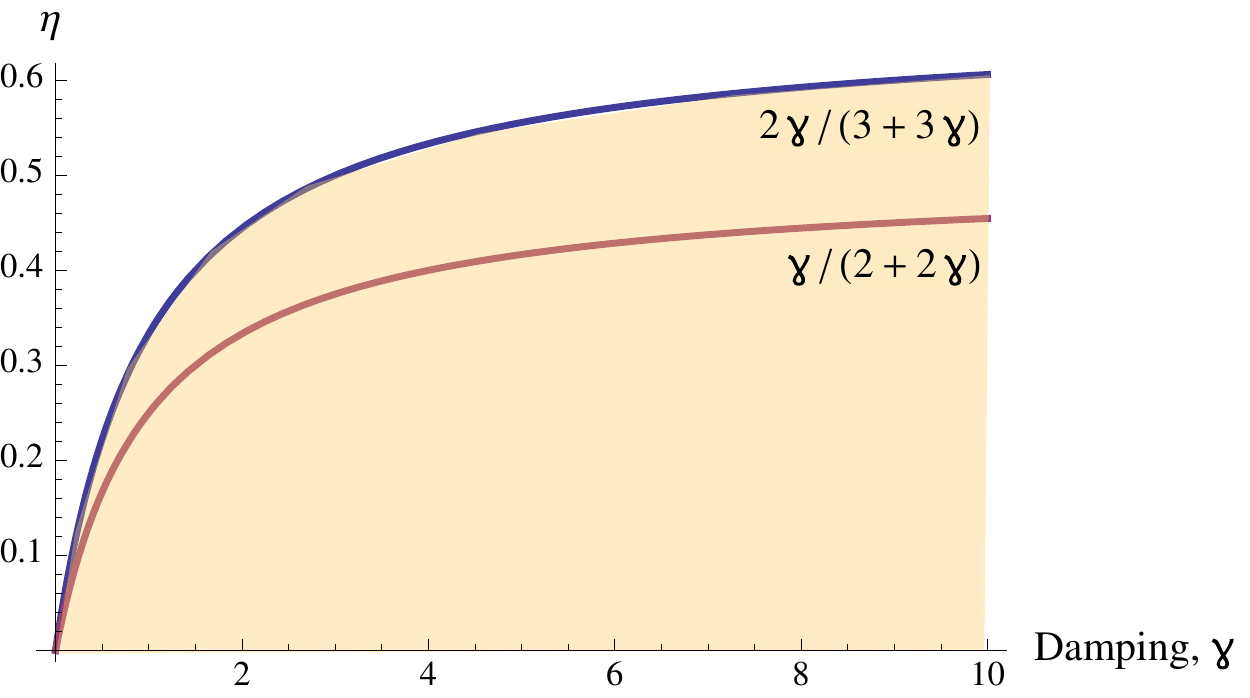}
	\caption{Region of validity $\eta\in[0,\bar\eta(\gamma)]$}\label{fig:ineq}
\end{figure}

\subsubsection{The case $\eta=0$} In this case, the interacting system reduces to a non-interacting system, see for example \cite{ExponentialRateOfConvergenceToEquilibrium} for the investigation of such systems. We recover known results on {\em hypocoercivity} with the explicit decay rate
\[
 \lambda(\gamma) = \gamma\bl 1 - \sqrt{\frac{\gamma^2}{4+\gamma^2}}\,\br,
\]
depicted in Figure~\ref{fig:case1}. The decay rate is maximal when $\gamma=(2(\sqrt{5}-1))^{1/2}$. 
\begin{figure}[h]
 \includegraphics[width=0.5\textwidth]{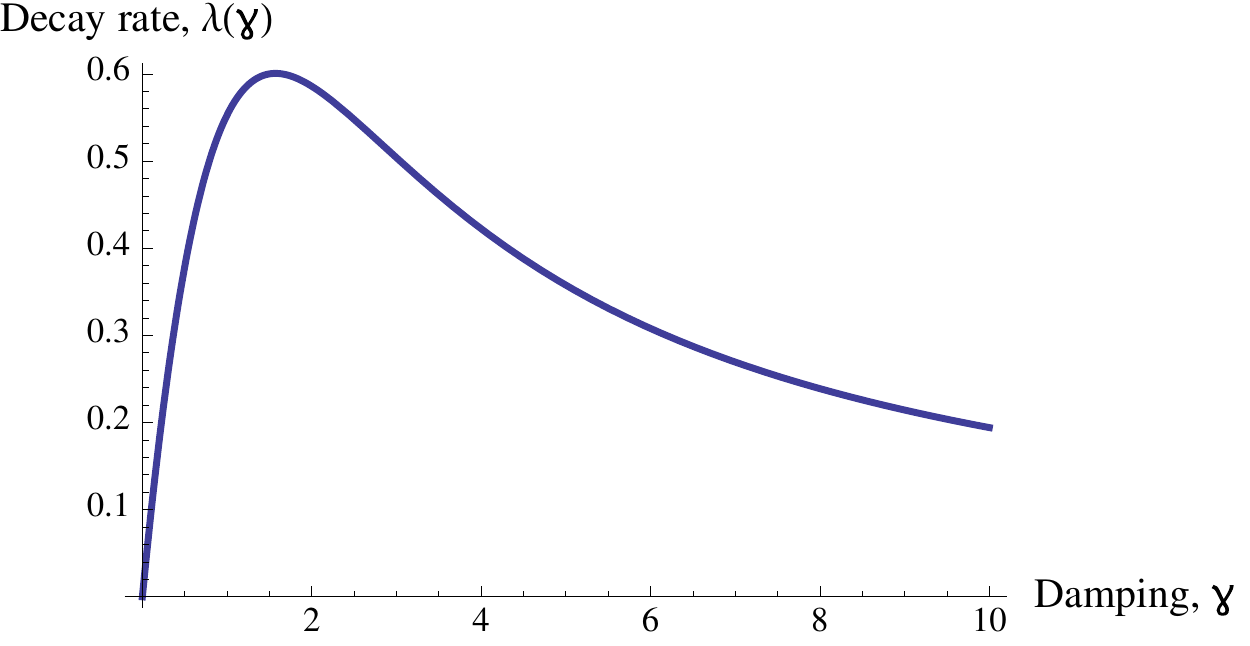}
 \caption{Rate of convergence without interaction $\eta\equiv 0$}\label{fig:case1}
\end{figure}

\subsubsection{The non-retarded case $\eta>0$} This case recovers the estimates shown in \cite{Bolley2010}, however, here we obtain the explicit form of $\lambda=\lambda(\gamma)$. In order to guarantee convergence towards the unique equilibrium state $\mu_\infty$, we choose 
$\eta=2\gamma/(3+3\gamma)$ and $\eta=\gamma/(2+2\gamma)$, which satisfies the condition on $\eta$ discussed above (cf.~Figure~\ref{fig:ineq}).
This  leads to the decay rates shown in Figure~\ref{fig:case2}, which strongly resembles the decay behavior seen in Figure~\ref{fig:case1}. One also notices the decrease in decay rates for increasing values of $\eta$, which clearly signifies the effect of the influence of interaction.
\begin{figure}[h]\centering
\includegraphics[width=0.5\textwidth]{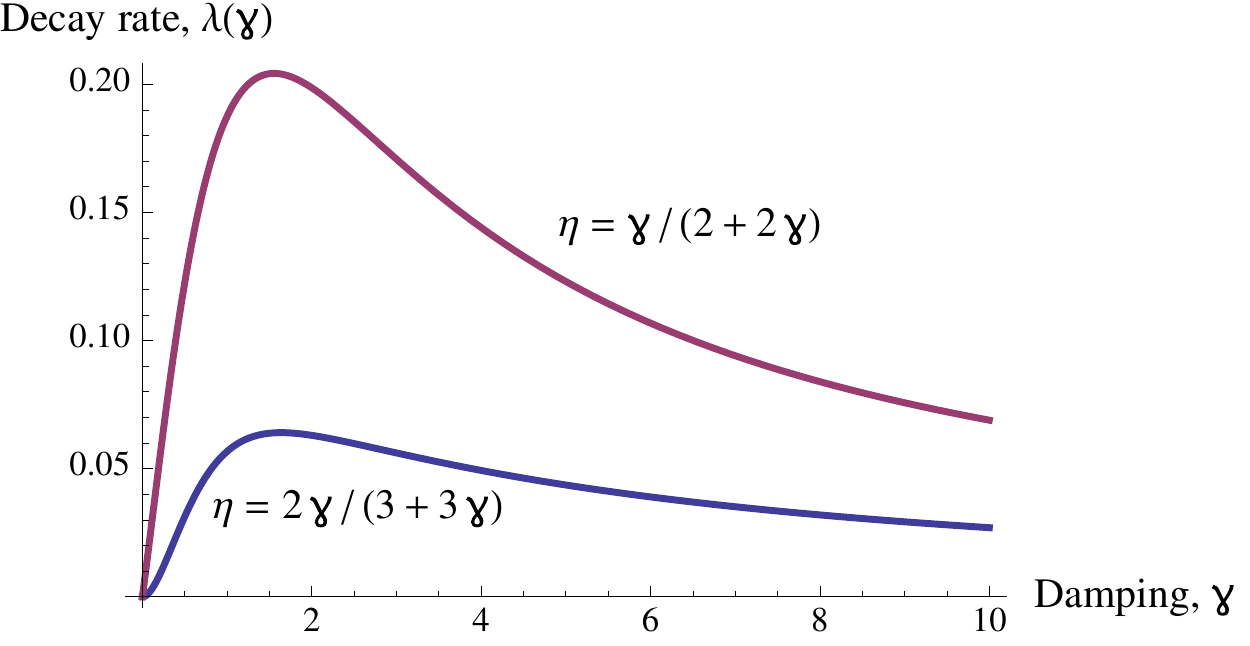}
\caption{Decay rates for various $\eta=\eta(\gamma)$, delay $H\equiv 0$}\label{fig:case2}
\end{figure}

\subsubsection{The retarded case $H\in(0,\infty)$} As expected, the decay rate decreases monotonically with increasing $H$ as seen in Figure~\ref{fig:case3}A. Here, we choose $\gamma=1$, $\eta=\gamma/(2+2\gamma)$ and vary $H\in(0,\infty)$. Similarly, Figure~\ref{fig:case3}B shows the decay rates for different delays $H$, with varying $\gamma$, where for each $\gamma$, we observe hypocoercivity trends in the decay rates as in Figure~\ref{fig:case2}.
\begin{figure}[h]\centering
\begin{subfigure}[b]{0.49\textwidth}
\centering
\includegraphics[width=\textwidth]{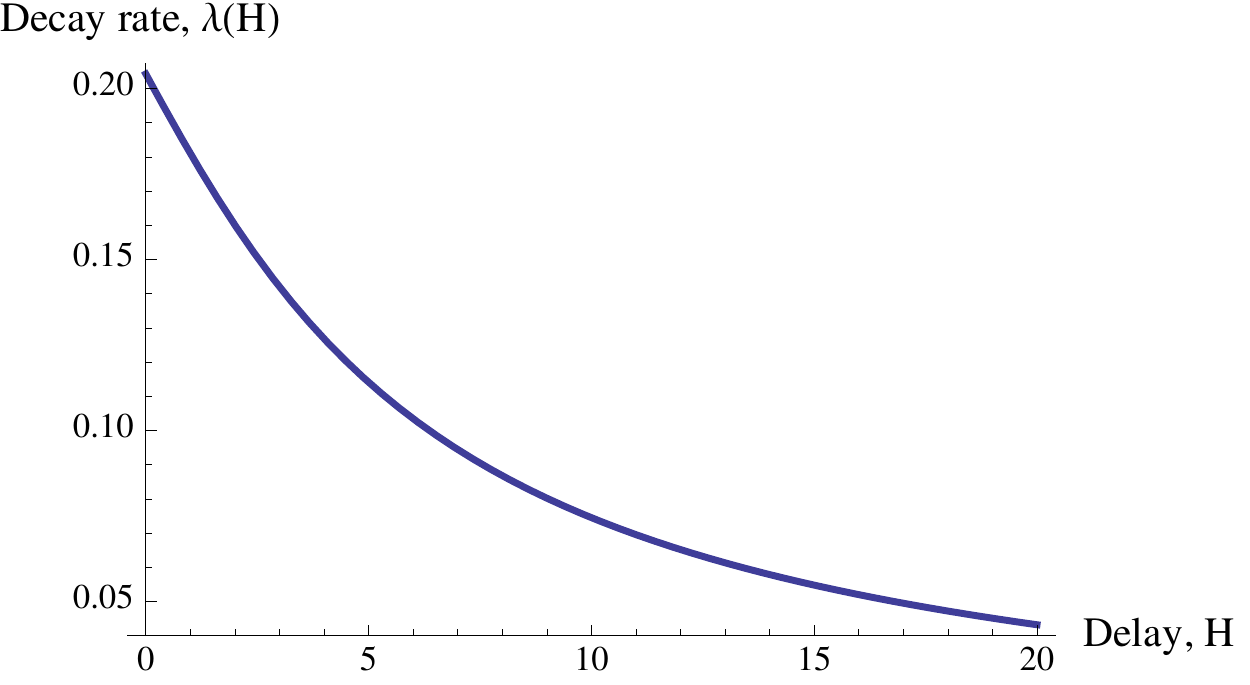}
\caption{Dependence of decay rate on delay $H$}
\end{subfigure}
\hfill
\begin{subfigure}[b]{0.49\textwidth}
\centering
\includegraphics[width=\textwidth]{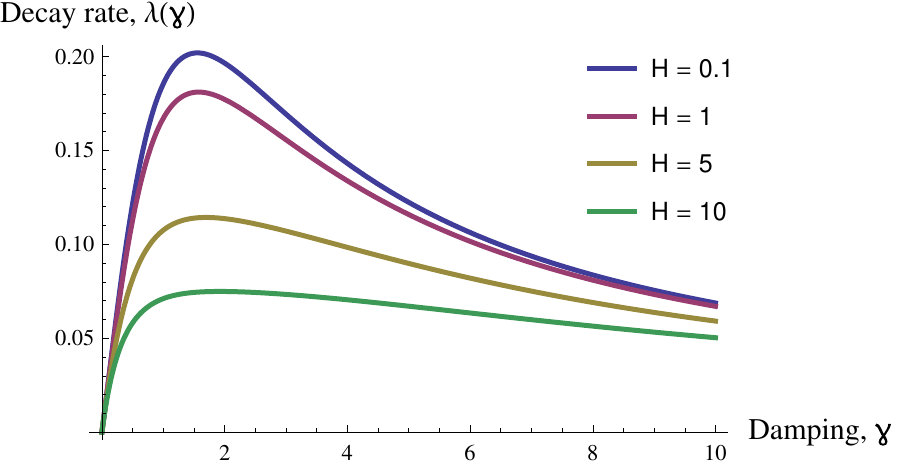}
\caption{Decay rates for various delays $H$}
\end{subfigure}
\caption{With delay $H>0$ and interaction $\eta(\gamma)=\gamma/(2+2\gamma)$}\label{fig:case3}
\end{figure}

\subsubsection{The retarded case $H=\infty$}
In this part, we prove a quantitative  convergence result for all solutions of the delay Vlasov--Fokker--Planck equation \eqref{eq:meanfield distribution} with an infinite cut-off, i.e., $h(t) = t$, associated to the microscopic system (\ref{eq:N interacting fibers}). As observed in Figure~\ref{fig:case3}A, an increase in $H$ reduces the exponential decay rate. In particular, we have that $\lambda\to 0$ as $H\to\infty$. Therefore, in this case, we cannot expect an exponential convergence rate, since the estimate \eqref{eq:ineq dq} given in Theorem~\ref{prop:decay} provides an upper bound for the decay. In fact, we will provide a polynomial (in time) convergence result, which we elucidate in the following discussion.

Our starting point is the differential inequality (\ref{eq:diff_inequality_int}). In the present situation, this reads
\begin{align}\label{eq:diff_inequality_int2}
 \frac{d}{dt}\J_t \le -\lambda_1 \J_t + \lambda_2\frac{1}{t} \int_{0}^t\J_s ds.
\end{align}
Consider  some function $\phi=\phi(t)$ satisfying the equality
\begin{align}\label{eq:kummerint}
 \frac{d\phi}{dt}(t) = -\lambda_1 \phi + \lambda_2\frac{1}{t} \int_{0}^t \phi(s) ds.
\end{align}
We further note that the statement in Proposition~\ref{prop halanay} remains valid, if $f(t,y(t),\sup_{s\in[t-H,t]}y(s))$ is substituted by $f(t,y(t)),(1/t)\int_0^t y(s)\,ds)$ (cf.~\cite{smith1987monotone}). Consequently, one obtains the estimate 
\[
 \J_t \le \phi(t)\qquad\text{for all\; $t\ge 0$}.
\]
Therefore, investigating the asymptotic behavior of equation (\ref{eq:kummerint}) gives us the desired decay property for the retarded mean-field equation \eqref{eq:meanfield distribution} with an infinite cut-off.

Reformulating equation (\ref{eq:kummerint}), we obtain
\begin{align}\label{eq:kummerprev}
\frac{d}{dt}\left(t \frac{d\phi}{dt}\right) =  -\lambda_1 \frac{d}{dt}\left(t \phi\right) + \lambda_2\phi,
\end{align}
or equivalently
\begin{align}\label{eq:kummerprev2}
 t \frac{d^2\phi}{dt^2} + (1+\lambda_1 t)\frac{d\phi}{dt} + (\lambda_1 - \lambda_2 ) \phi =0.
\end{align}
Substituting $u(t) =e^{\lambda_1 t} \phi(t)$ and rescaling time $\tau = \lambda_1 t$ leads to the so-called {\em Kummer's equation}
\begin{align}\label{eq:kummer}
 \tau \frac{d^2u}{d\tau^2} +  (1- \tau) \frac{du}{d\tau} - \frac{\lambda_2}{\lambda_1}u =0,
\end{align}
where $\Lambda = \lambda_2/\lambda_1\in[0,1)$. Kummer's equation is known to be solved by a linear combination of two {\em confluent hypergeometric functions} \cite{abramowitz1964handbook}. To satisfy our initial condition $u(\tau_0)=y_0$, we obtain a particular representation of the solution $u$ in the form of a generalized hypergeometric series
\[
 u(\tau) = M(\Lambda,1,\tau)\,y_0 = \left(\sum\nolimits_{n=0}^\infty \Lambda^{(n)}\frac{\tau^n}{n!}\right) y_0,
\]
where $\Lambda^{(0)}=1$ and $\Lambda^{(n)}=\Lambda(\Lambda+1)(\Lambda+2)\cdots (\Lambda + n -1)$ for $n\ge 1$, is the rising factorial. This function is also known as {\em Kummer's function}. 

For the present parameters, we obtain an asymptotic behavior of the form
\[
 u(\tau) = \frac{1}{\Gamma(\Lambda)}\tau^{\Lambda-1}e^{\tau} + \frac{(-1)^{-\Lambda}}{\Gamma(1-\Lambda)}\tau^{-\Lambda} + \mathcal{O}(\tau)\qquad\text{as\; $\tau\to\infty$}.
\]
%
Transforming the above substitutions backward we obtain an asymptotic behavior for $\phi$
\[
 \phi(t) \sim \frac{\lambda_1^{\Lambda-1}}{\Gamma(\Lambda)} t^{\Lambda-1} = \frac{\lambda_1^{\Lambda-1}}{\Gamma(\Lambda)} \frac{1}{t^{1-\Lambda}}\qquad\text{as\; $t\to\infty$}.
\]
In this way we obtain a polynomial decay, in contrast to the exponential decay for finite cut-off.


\section{Conclusion and outlook}

In this paper, we provided a systematic approach to determining explicit decay rates for general delay Vlasov--Fokker--Planck equation of the form \eqref{eq:meanfield distribution}. Under the conditions stated in Theorems~\ref{prop:decay} and \ref{thm:limit}, the mean-field equation is known to be ergodic and possesses a unique stationary state, whereby exponential rate of convergence is obtained when $H<\infty$. On the other hand, one can only expect polynomial decay to equilibrium when $H=\infty$, i.e., when all the history of the solution paths are taken into consideration. 

As indicated in the introduction, we are especially interested in applying the techniques developed in this paper to the delay mean-field fiber equations discussed in \cite{klar}. Unfortunately, the methods here do not trivially translate to different state spaces such as $\R^d\times \mathbb{S}^{d-1}$, where $\mathbb{S}^{d-1}$ denotes the unit sphere in $\R^d$, since the distances involved are not  purely euclidean in nature.

\appendix
\section{Proof of Proposition~\ref{prop:wellposedness}}\label{proof:wellposedness}
 The proof of the statement follows a Picard type iteration procedure. 
 
 Let $\nu,\hat\nu\in\C([0,T],\P_2(\M))$ be arbitrary and $Z_t,\hat Z_t$ be two stochastic processes satisfying the stochastic differential equations
 \begin{subequations}\label{eq:aux_stochastic}
   \begin{align}
    d\X &= \V dt\\
    d\V &= A(\X) dt + \bl\frac{1}{h(t)}\int_{t-h(t)}^{t} \int_{\M} B(\X,\hat x)\eta_s(d\hat x,d\hat v) \,d s\br dt -\gamma\V d t + \sqrt{2\sigma} d W_t,
   \end{align}
 \end{subequations}
 with $\eta\in\{\nu,\hat\nu\}$ respectively. Notice that we have chosen the same Brownian motion $\{W_t\}$ for both processes. Under Assumptions~\ref{ass}, it is not difficult to see that \eqref{eq:aux_stochastic} admits continuous pathwise unique and adapted sample paths for any $\eta\in \C([0,T],\P_2(\M))$ (cf.~\cite{mohammed, von2010existence}).
 
 Consequently, we may define 
 \[
  \mu_t = \text{law}(\Z),\, \hat\mu_t = \text{law}(\hat\Z)\,\in\, \C([0,T],\P_2(\M)),
 \]
 which in turn defines a self-mapping $\mathcal{T}\colon \nu\mapsto \mu$ within the complete metric space $\C([0,T],\P_2(\M))$. Denoting the difference of the solutions by $x_t = \X-\hat \X$, $v_t = \V - \hat \V$, we have that $z_t=(x_t,v_t)$ satisfies \eqref{eq:difference}. Therefore, applying the It\^o formula yields
 \begin{align*}
  \frac{1}{2}\frac{d}{dt}\E\lb |z_t|^2 \rb &= \E\lb\langle x_t,v_t\rangle\rb + \E\lb\langle v_t, A(\X) - A(\hat \X)\rangle\rb\\
  &\hspace*{6em} - \frac{1}{h(t)}\int_{t-h(t)}^{t} \E\lb\langle v_t,\K[\mu_t](\X) - \K[\hat\mu_t](\hat\X)\rangle\rb d s - \gamma \E\lb|v_t|^2\rb \\
  &\le \frac{1}{2}\bl c_1\E\lb |x_t|^2 \rb +  c_2\E\lb|v_t|^2\rb + \frac{c_B}{h(t)}\int_{t-h(t)} ^t \iint |y_1-\hat y_1|^2 \nu_s(dy)\otimes\hat\nu_s(d\hat y)\, ds \br,
 \end{align*}
 we obtain the differential inequality
 \[
  \frac{d}{dt}\E\lb |z_t|^2 \rb \le c_0 \E\lb |z_t|^2 \rb + c_B\frac{1}{h(t)}\int_{t-h(t)}^t \iint |y_1-\hat y_1|^2 \nu_s(dy)\otimes\hat\nu_s(d\hat y)\, ds,
 \]
Following the arguments made in \cite{klar}, we construct a solution by the Picard iteration procedure, and show that the generated sequence converges. For this reason, we define the sequence $(\mu^{(k)})$ recursively by $\mu^{(k+1)}=\mathcal{T}\mu^{(k)}$ for $k\ge 0$ and denote
\[
 \mathcal{E}_k(t) = \E\lb \big|Z_t^{(k+1)}-Z_t^{(k)}\big|^2 \rb.
\]
We begin by considering the case $h(t)=t$. In this case, we have the estimate
\begin{align*}
 \frac{d}{dt}\mathcal{E}_k(t) \le c_0 \mathcal{E}_k(t) + c_B\frac{1}{t}\int_{0}^t \mathcal{E}_{k-1}(s)\,ds,
\end{align*}
Integrating over time $t\ge0$ and fixing the initial condition $Z_0^{(k)}=z$ for all $k\ge 0$, we have
\[
 \mathcal{E}_k(t) \le c_0 \int_0^t \mathcal{E}_k(s)\,ds + c_B\int_0^t\frac{1}{s}\int_{0}^s \mathcal{E}_{k-1}(\sigma)\,d\sigma\,ds
\]
Setting $g(t,s):=\ln(t)-\ln(s)$, we compute recursively for $\mathcal{E}_k$ to obtain
\begin{align*}
 \mathcal{E}_k(t) &\le  c_Be^{c_0t}\int_0^t (e^{-c_0 s}/s)\int_{0}^s \mathcal{E}_{k-1}(\sigma)\,d\sigma\,ds \le c_Be^{c_0t}\int_0^t g(t,t_1) \,\mathcal{E}_{k-1}(t_1)\,dt_1 \\
 &\le \big(c_B e^{c_0 t}\big)^2\int_0^t g(t,t_1)\int_0^{t_2}g(t_1,t_2)\, \mathcal{E}_{k-1}(t_2)\,dt_2\,dt_1\\
 &\cdots \\
 &\le \big(c_B e^{c_0 t}\big)^k\lb \int_0^t g(t,t_1)\cdots \int_0^{t_{k-1}} g(t_{k-1},t_k)\,dt_k\cdots dt_1\rb\sup\nolimits_{t_k\in[0,T]}\mathcal{E}_0(t_k),
\end{align*}
where in the second inequality, we used the integration by parts formula. Elementary computations of the terms in the bracket gives
\[
 \int_0^t g(t,t_1)\cdots \int_0^{t_{k-1}} g(t_{k-1},t_k)\,dt_k\cdots dt_1 \le \frac{1}{k!}t^k.
\]
Furthermore, note that by definition, we have
\[
 \dist_2(\mu_t^{(k+1)},\mu_t^{(k)})\le\mathcal{E}_k(t)\qquad\text{for all $k\ge0$}.
\]
Therefore, summing up the terms in $k\ge 0$ yields
\[
 \sum\nolimits_{k\ge 0} \dist_2(\mu_t^{(k+1)},\mu_t^{(k)}) \le e^{c t}\sup\nolimits_{t_k\in[0,T]}\mathcal{E}_0(t_k) <+\infty,
\]
with $c=c_B\exp(c_0 T)$, and hence, $\dist_2(\mu_t^{(k+1)},\mu_t^{(k)})\to 0$ for every $t> 0$. Consequently,  the Picard sequence $(\mu^{(k)})$ converges uniformly in $\C([0,T],\P_2(\M))$ to the solution of \eqref{eq:general_meanfield}. Since $T>0$ was chosen arbitrarily, the solution may be extended to the interval $[0,\infty)$.

As for uniqueness, we take two solutions $\mu$ and $\hat \mu$ and denote the difference
\[
 \mathcal{E}(t) = \E\lb \big|Z_t-\hat Z_t\big|^2 \rb.
\]
From the estimate above, we obtain
\[
 \mathcal{E}(t) \le c\int_0^t g(t,s)\,\mathcal{E}(s)\,ds.
\]
Following the computations above, we obtain for some $k_0\in\N$ sufficiently large,
\[
 0\le \bl 1 - \frac{c^k t^k}{k!}\br \sup\nolimits_{t\in[0,T]} \mathcal{E}(t) \le 0\qquad\text{for every $k\ge k_0$}.
\] 
Passing to the limit $k\to\infty$ yields $\mathcal{E}=0$ on $[0,T]$, and hence, $\dist_2(\mu_t,\hat\mu_t)=0$ for all $t\in[0,T]$.

Mimicking the arguments made above, we may obtain well-posedness also for the other cases of $h(t)$. We omit the proof, since it bears similarities to the proof given in \cite{klar}.\hfill$\Box$

\bibliographystyle{plain}
\bibliography{references}

\begin{thebibliography}{10}

\bibitem{abramowitz1964handbook}
Milton Abramowitz and Irene~A Stegun.
\newblock {\em Handbook of mathematical functions: with formulas, graphs, and
  mathematical tables}.
\newblock Number~55. Courier Corporation, 1964.

\bibitem{BZ}
K.~Beauchard and E.~Zuazua.
\newblock Large time asymptotics for partially dissipative hyperbolic systems.
\newblock {\em Arch. Rational Mech. Anal.}, 199:177--227, 2011.

\bibitem{nat}
S.~Bianchini, B.~Hanouzet, and R.~Natalini.
\newblock Long-time effect of relaxation for hyperbolic conservation laws.
\newblock {\em Commun. on Pure and Appl. Math.}, 60(11):1559--1622, 2007.

\bibitem{completemetricspace}
F.~Bolley.
\newblock Separability and completeness for the wasserstein distance.
\newblock In {\em S{\'{e}}minaire de probabilit{\'{e}}s XLI, Lecture Notes
  Math. 1934}, pages 371--377. Springer Berlin Heidelberg, 2008.

\bibitem{Bolley2010}
F.~Bolley, A.~Guillin, and F.~Malrieu.
\newblock Trend to equilibrium and particle approximation for a weakly
  selfconsistent vlasov-fokker-planck equation.
\newblock {\em ESAIM: Mathematical Modelling and Numerical Analysis},
  44(5):867--884, 2010.

\bibitem{BonillaGKMW2008}
L.~Bonilla, T.~G\"{o}tz, A.~Klar, N.~Marheineke, and R.~Wegener.
\newblock Hydrodynamic limit of a fokker--planck equation describing fiber
  lay–down processes.
\newblock {\em SIAM Journal on Applied Mathematics}, 68(3):648–665, 2008.

\bibitem{klar}
R.~Borsche, A.~Klar, C.~Nessler, A.~Roth, and O.~Tse.
\newblock A retarded mean-field approach for interacting fiber structures.
\newblock {\em Preprint, arXiv:1501.06465}, 2015.

\bibitem{Carrillo:2009:10.3934/krm.2009.2.363}
J.A. Carrillo, M.R. D'Orsogna, and V.~Panferov.
\newblock Double milling in self-propelled swarms from kinetic theory.
\newblock {\em Kinetic and Related Models}, 2:363--378, 2009.

\bibitem{Carrillo:2010:10.1142/S0218202510004684}
J.A. Carrillo, A.~Klar, S.~Martin, and S.~Tiwari.
\newblock Self-propelled interacting particle systems with roosting force.
\newblock {\em Mathematical Models and Methods in Applied Sciences},
  20:1533--1552, 2010.

\bibitem{carrillo}
J.A. Carrillo and G.~Toscani.
\newblock Contractive probability metrics and asymptotic behavior of
  dissipative kinetic equations.
\newblock {\em Riv. Mat. Univ. Parma}, 6:75–198, 2007.

\bibitem{Che}
I.~L. Chern.
\newblock Asymptotic behavior of smooth solutions for partially dissipative
  hyperbolic systems with a convex entropy.
\newblock {\em Commun. Math. Phys.}, 172:39--55, 1995.

\bibitem{CG07}
J.~F. Coulombel and T.~Goudon.
\newblock The strong relaxation limit of the multidimensional isothermal
  {E}uler equations.
\newblock {\em Trans. Am. Math. Soc.}, 359(2):637--648, 2007.

\bibitem{Desvillettes}
L.~Desvillettes and C.~Villani.
\newblock On the trend to global equilibrium in spatially inhomogeneous
  entropy-dissipating systems: The linear fokker-planck equation.
\newblock {\em Communications on Pure and Applied Mathematics}, 54(1):1--42,
  2001.

\bibitem{ExponentialRateOfConvergenceToEquilibrium}
J.~Dolbeault, A.~Klar, C.~Mouhot, and C.~Schmeiser.
\newblock Exponential rate of convergence to equilibrium for a model describing
  fiber lay-down processes.
\newblock {\em Applied Mathematics Research eXpress (Vol. 2013)}, pages
  165--175, 2013.

\bibitem{Dolbeault2009511}
J.~Dolbeault, C.~Mouhot, and C.~Schmeiser.
\newblock Hypocoercivity for kinetic equations with linear relaxation terms.
\newblock {\em Comptes Rendus Math{\'e}matique}, 347(9-10):511 -- 516, 2009.

\bibitem{Dolbeault2015}
J.~Dolbeault, C.~Mouhot, and C.~Schmeiser.
\newblock Hypocoercivity for linear kinetic equations conserving mass.
\newblock {\em Transactions of the American Mathematical Society},
  367(6):3807--3828, 2015.

\bibitem{golse}
F.~Golse.
\newblock On the dynamics of large particle systems in the mean field limit,
  2013.

\bibitem{Golse2003}
François Golse.
\newblock The mean-field limit for the dynamics of large particle systems.
\newblock {\em Journées équations aux dérivées partielles}, pages 1--47,
  2003.

\bibitem{Golse2012}
François Golse.
\newblock The mean-field limit for a regularized vlasov-maxwell dynamics.
\newblock {\em Commun. Math. Phys.}, pages 789--816, 2012.

\bibitem{GoetzKMW2007}
T.~G\"{o}tz, A.~Klar, N.~Marheineke, and R.~Wegener.
\newblock A stochastic model and associated fokker–planck equation for the
  fiber lay-down process in nonwoven production processes.
\newblock {\em SIAM Journal on Applied Mathematics}, 67(6):1704–1717, 2007.

\bibitem{grothaus}
Martin Grothaus and Axel Klar.
\newblock Ergodicity and rate of convergence for a nonsectorial fiber lay-down
  process.
\newblock {\em {SIAM} J. Math. Analysis}, 40(3):968--983, 2008.

\bibitem{halanay}
A.~Halanay.
\newblock {\em Differential equations: Stability, oscillations, time lags}.
\newblock Academic Press, New York-London, 1966.

\bibitem{3dfiber}
A.~Klar, J.~Maringer, and R.~Wegener.
\newblock A 3d model for fiber lay-down in non-woven production processes.
\newblock {\em Mathematical Models and Methods in Applied Sciences},
  22(09):1250020, 2012.

\bibitem{ApproximateModels}
A.~Klar, F.~Schneider, and O.~Tse.
\newblock Approximate models for stochastic dynamic systems on the sphere and
  associated fokker-planck equations.
\newblock {\em Kinetic and Related Models}, 7(3):509 -- 529, 2014.

\bibitem{klar2015entropy}
A~Klar and O~Tse.
\newblock An entropy functional and explicit decay rates for a nonlinear
  partially dissipative hyperbolic system.
\newblock {\em ZAMM-Journal of Applied Mathematics and Mechanics/Zeitschrift
  f{\"u}r Angewandte Mathematik und Mechanik}, 95(5):469--475, 2015.

\bibitem{mohammed}
S.E.A. Mohammed.
\newblock {\em Stochastic Functional Differential Equations}.
\newblock Pitman, Boston, 1984.

\bibitem{neunzert}
H.~Neunzert.
\newblock The vlasov equation as a limit of hamiltonian classical mechanical
  systems of interacting particles.
\newblock {\em Trans. Fluid Dynamics}, 18:663--678, 1977.

\bibitem{smith1987monotone}
Hal Smith.
\newblock Monotone semiflows generated by functional differential equations.
\newblock {\em Journal of Differential Equations}, 66(3):420--442, 1987.

\bibitem{stochasticHamiltonianDissipativeSystems}
D.~Talay.
\newblock Stochastic hamiltonian dissipative systems: Exponential convergence
  to the invariant measure, and discretization by the implicit euler scheme.
\newblock {\em Markov Processes and Related Fields}, 8(2):163--198, 2002.

\bibitem{optimaltransport:villani}
C.~Villani.
\newblock Grundlehren der mathematischen wissenschaften.
\newblock In {\em Optimal transport, old and new}, volume 338. Springer-Verlag
  Berlin Heidelberg, 2009.

\bibitem{hypocoercivity:villani}
C.~Villani.
\newblock Hypocoercivity.
\newblock {\em Memoirs of the American Mathematical Society}, 202(950), 2009.

\bibitem{von2010existence}
Max-K von Renesse and Michael Scheutzow.
\newblock Existence and uniqueness of solutions of stochastic functional
  differential equations.
\newblock {\em Random Operators and Stochastic Equations}, 18(3):267--284,
  2010.

\end{thebibliography}

\end{document}